\documentclass[10pt]{article}

\usepackage{t1enc}
\usepackage[latin1]{inputenc}
\usepackage[english]{babel}
\usepackage{amsmath,amsthm}
\usepackage{amsfonts}
\usepackage{latexsym}
\usepackage{graphicx}
\usepackage{txfonts}
\usepackage[natural]{xcolor}
\usepackage{dsfont}
\usepackage{lineno}

\newtheorem{thm}{Theorem}
\newtheorem{lem}[thm]{Lemma}
\newtheorem{cor}[thm]{Corollary}

\newtheorem{defn}[thm]{Definition}

\makeatletter

\newcommand{\Rmnum}[1]{\expandafter\@slowromancap\romannumeral #1@}
\makeatother

\begin{document}


\title{A note on the vertex arboricity of signed graphs\thanks{This project is supported by the Natural Science Basic Research Plan in Shaanxi Province of China (No.\,2017JM1010), the SRFDP (No.\,20130203120021), the NSFC (No.\,11301410), and the Fundamental Research Funds for the Central Universities (No.\,JB170706).}}
\author{Weichan Liu, Chen Gong, Lifang Wu, Xin Zhang\thanks{Corresponding author. Email address: xzhang@xidian.edu.cn.}\\
{\small School of Mathematics and Statistics, Xidian University, Xi'an 710071, P.\,R.\,China}}
\date{}

\maketitle

\begin{abstract}\baselineskip  0.395cm
A signed tree-coloring of a signed graph $(G,\sigma)$ is a vertex coloring $c$ so that $G^{c}(i,\pm)$ is a forest for every $i\in c(u)$ and $u\in V(G)$, where $G^{c}(i,\pm)$ is the subgraph of $(G,\sigma)$
whose vertex set is the set of vertices colored by $i$ or $-i$ and edge set is the set of positive edges with two end-vertices colored both by $i$ or both by $-i$, along with
the set of negative edges with one end-vertex colored by $i$ and the other colored by $-i$. If $c$ is a function from $V(G)$ to $M_n$, where $M_n$ is $\{\pm 1,\pm 2,\ldots,\pm k\}$ if $n=2k$, and $\{0,\pm 1,\pm 2,\ldots,\pm k\}$ if $n=2k+1$, then $c$ a signed tree-$n$-coloring of $(G,\sigma)$. The minimum integer $n$ such that $(G,\sigma)$ admits a signed tree-$n$-coloring is the signed vertex arboricity of $(G,\sigma)$, denoted by $va(G,\sigma)$. In this paper, we first show that two switching equivalent signed graphs have the same signed vertex arboricity, and then
prove that $va(G,\sigma)\leq 3$ for every balanced signed triangulation and for every edge-maximal $K_5$-minor-free graph with balanced signature. This generalizes the well-known result that the vertex arboricity of every planar graph is at most 3.\\[.2em]
\end{abstract}

\baselineskip  0.395cm

\section{Introduction}

\newcommand{\la}{\mathrm{la}}

In this paper, all graphs are finite, simple and undirected. Let $V(G)$ and $E(G)$ be the vertex set and the edge set of $G$, respectively. We denote by $\Delta(G)$ and $\delta(G)$ the maximum and minimum degree of $G$, respectively.

A \emph{signed graph} is a graph in which each edge has a positive or negative signature. Precisely, a signed graph $(G,\sigma)$ is a graph $G$ with a signature $\sigma: E(G)\rightarrow \{1,-1\}$. Edges with signature $1$ are \emph{positive} edges while edges with signature $-1$ are \emph{negative} edges.
A graph with no signature is usually called an \emph{unsigned graph}.
Signed graphs appeared first in a mathematical paper of Harary \cite{Harary} in 1955, and have been rediscovered many times because they come up naturally in many unrelated areas \cite{Zaslavsky1998,Zaslavsky2012}. For example, signed graphs have been used in social psychology \cite{CH} to model social situations, with positive edges representing friendships and negative edges representing enmities between nodes, which represent people.

This paper focuses in the coloring problems of sighed graphs. First, let us look at the proper coloring of a signed graph. Actully, there are two definitions on this in the history.

In an early paper, Zaslavsky \cite{Zaslavsky1882}
defines a \emph{(signed) coloring} of a signed graph $(G,\sigma)$ with $k$ colors or with $2k+1$ signed colors to be a mapping  $c:V(G)\rightarrow\{-k,-(k-1),\ldots,-1,0,1,\ldots,(k-1),k\}$ such that
for every edge $uv\in E(G)$, $c(u)\neq c(v)$ if $\sigma(uv)=1$, and $c(u)\neq -c(v)$ if $\sigma(uv)=-1$. Recently in 2016, M\'a\v cajov\'a, Raspaud and \v Skoviera \cite{MRS} pointed out that Zaslavsky's definition has disadvantage that this definition for signed graphs does not directly transfer from the definition for unsigned graphs. Thereby, they diverged from the above definition adopted by Zaslavsky and proposed the following definition, which aligns the definitions for both unsigned and signed versions.

Let $M_n\subseteq \mathds{Z}$, where $M_n$ is $\{\pm 1,\pm 2,\ldots,\pm k\}$ if $n=2k$, and $\{0,\pm 1,\pm 2,\ldots,\pm k\}$ if $n=2k+1$. A mapping $c:V(G)\rightarrow M_n$ such that
$c(u)\neq \sigma(uv)c(v)$ for every edge $uv\in E(G)$ is a \emph{(signed) $n$-coloring} of a signed graph $(G,\sigma)$. The \emph{signed chromatic number} of $(G,\sigma)$, denoted by $\chi(G,\sigma)$, is the smallest integer $n$ such that $(G,\sigma)$ admits an $n$-coloring.

Define the all-positive signed graph as having all positive edges. Since the signed coloring rules for such a graph are equivalent to the ones for an unsigned graph, its signed chromatic number is the same with its (unsigned) chromatic number. Hence the proper coloring of signed graphs generalizes the one of unsigned graphs. In view of this, it is natural to generalize the improper coloring of unsigned graphs to its signed version.

A \emph{tree-$k$-coloring} of an unsigned graph $G$ is a function $\varphi$ from the vertex set $V(G)$ to $\{1,2,\ldots,k\}$ so that the graph induced by $\varphi^{-1}(i)$ is a union of trees for every $1\leq i\leq k$. The minimum integer $k$ so that $G$ admits a tree-$k$-coloring is the \emph{vertex arboricity} of $G$, denoted by $va(G)$. The notion of vertex arboricity was introduced by Chartrand, Kronk and Wall \cite{CKW} in 1968.

The aim of this paper is to investigate the vertex arboricity of signed graphs, i.e., investigate the signed version of the vertex arboricity.
Before doing this, we shall properly define the tree-coloring of a signed graph.
In order to complete this work, we first introduce an equivalent definition
for the (signed) $n$-coloring of a signed graph $(G,\sigma)$.

Mapping the vertex set $V(G)$ of a signed graph $(G,\sigma)$ into $M_n$, we then obtain a (signed) $n$-coloring $c$ of $(G,\sigma)$.

Let $G^{c}(i,\pm)$ be the subgraph of $(G,\sigma)$
whose vertex set is the set of vertices colored by $i$ or $-i$ and edge set is the set of positive edges with two end-vertices colored both by $i$ or both by $-i$, along with
the set of negative edges with one end-vertex colored by $i$ and the other colored by $-i$. Clearly, $G^{c}(i,\pm)=G^{c}(-i,\pm)$.


By the definitions of $G^{c}(i,\pm)$, one can easily conclude that $c$ is a proper coloring of $(G,\sigma)$ if and only if $G^{c}(i,\pm)$ is an empty graph (i.e., graph with no edge) for every $i\in M_n$.
Follow this idea, we can naturally define one another kind of coloring of the signed graph $(G,\sigma)$ as follows.

\begin{defn}
A signed tree-coloring of $(G,\sigma)$ is a vertex coloring $c$ so that $G^{c}(i,\pm)$ is a forest for every $i\in c(u)$ and $u\in V(G)$. If $c$ is a function from $V(G)$ to $M_n$, then we call $c$ a (signed) tree-$n$-coloring of $(G,\sigma)$. The minimum integer $n$ such that $(G,\sigma)$ admits a signed tree-$n$-coloring is the signed vertex arboricity of $(G,\sigma)$, denoted by $va(G,\sigma)$.
\end{defn}


\section{Main results and their proofs}

\emph{Switching} a vertex in a signed graph $(G,\sigma)$ means negating the signs of all the edges incident with that vertex. Switching a set of vertices means negating all the edges that have one end in that set and one end in the complementary set. Switching a series of vertices, once each, is the same as switching the whole set at once.

If a signed graph $(G,\sigma')$ is obtained from a signed graph $(G,\sigma)$ by a series of switchings, then we say that $(G,\sigma')$ and $(G,\sigma)$ are \emph{switching equivalent}.

\begin{thm}\label{switching}
  If $(G,\sigma)$ and $(G,\sigma')$ are switching equivalent, then $(G,\sigma)$ is tree-$n$-colorable iff $(G,\sigma')$ is tree-$n$-colorable, hence $va(G,\sigma)=va(G,\sigma')$.
\end{thm}

\begin{proof}
  We just need prove this result if $(G,\sigma')$ is obtained from $(G,\sigma)$ by switching one vertex $u$. Let $c$ be a tree-$n$-coloring of $(G,\sigma)$ and let $c(u)=i$. In what follows we claim that if we recolor $u$ by $-i$ then we would get a tree-$n$-coloring $c'$ of $(G,\sigma')$

  If $j\in M_n\setminus \{i,-i\}$, then $G^{c'}(j,\pm)$ is a forest, since any edge incident with $u$ is not contained in $G^{c'}(j,\pm)$ and thus $G^{c'}(j,\pm)=G^{c}(j,\pm)$.

  If $j=-i$, then we suppose, to the contrary, that $G^{c'}(-i,\pm)$ contains a cycle $C$. If $u\not\in V(C)$, then any edge incident with $u$ is not an edge of $C$ and thus $C$ is cycle in $G^c(-i,\pm)$, a contradiction. Hence we assume that $uv,uw\in E(C)$. Let $P$ be path derived from $C$ by deleting $uv$ and $uw$. Since $c(x)=c'(x)$ for every $x\in V(P)$ and $\sigma(e)=\sigma'(e)$ for every $e\in E(P)$, $P$ is a subgraph of $G^c(-i,\pm)$.
  By the definition of $G^{c'}(-i,\pm)$, we have $c'(v)=-i\sigma'(uv)$ and $c'(w)=-i\sigma'(uw)$. Note that $-i=c'(u)$. Since $-\sigma'(uv)=\sigma(uv)$ and $-\sigma'(uw)=\sigma(uw)$, $c(v)=c'(v)=i\sigma(uv)=c(u)\sigma(uv)$ and $c(w)=c'(w)=i\sigma(uw)=c(u)\sigma(uw)$. This implies that $uv$ and $uw$ are two edges in $G^c(-i,\pm)$. Therefore, $C$ becomes a cycle in $G^c(-i,\pm)$, a contradiction.
  Hence $G^{c'}(j,\pm)=G^{c'}(-i,\pm)$ is a forest.

 If $j=i$, then $G^{c'}(i,\pm)=G^{c'}(-i,\pm)$ is a forest by the above result.


 In conclusion, we prove that $G^{c'}(j,\pm)$ is a forest for any $j\in M_n$. Therefore, $c'$ is a tree-$n$-coloring of $(G,\sigma')$.
\end{proof}

Actually, Theorem \ref{switching} grantees that Definition 1 is well-defined.

A cycle of a signed graph $(G,\sigma)$ is \emph{positive} if it has an even number of negative edges. A singed graph is said to be \emph{balanced} if all of its cycles are positive.

By $G_1\cap G_2$ (resp.\,$G_1\cup G_2$), we denote the graph with vertex set $V(G_1)\cap V(G_2)$ (resp.\,$V(G_1)\cup V(G_2)$) and edge set $E(G_1)\cap E(G_2)$ (resp.\,$E(G_1)\cup E(G_2)$).

\begin{lem}\label{comb-lem}
Let $(G_1,\sigma_1)$ and $(G_2,\sigma_2)$ be balanced signed graphs such that $\sigma_1(e)=\sigma_2(e)$ if $e\in E(G_1)\cap E(G_2)$, and let $G=G_1\cup G_2$ be a graph with signature $\sigma$ such that $\sigma(e)=\sigma_1(e)$ if $e\in E(G_1)$ and $\sigma(e)=\sigma_2(e)$ if $e\in E(G_2)$. Let $c_1$ and $c_2$ be signed tree-colorings of $(G_1,\sigma_1)$ and $(G_2,\sigma_2)$, respectively, such that $c_1(v)=c_2(v)$ for any $v\in V(G_1)\cap V(G_2)$.
If $G_1\cap G_2=K_2$ or $K_3$, then combining $c_1$ with $c_2$ we obtain a signed tree-coloring $c$ of $(G,\sigma)$.
\end{lem}

\begin{proof}
Fist, we consider the case that $G_1\cap G_2=K_2=:xy$. One can image that $G_1$ is drawn on one side of the edge $xy$ while $G_2$ is drawn on the other side of the edge $xy$, so $xy$ is non-crossed in this drawing.

If $c$ is not a signed tree-coloring of $(G,\sigma)$, then $G^c(i,\pm)$ contains a cycle $C$ for some color $i$.
Clearly, $x,y\in V(C)$, $c(x),c(y)\in\{i,-i\}$, and the graph $C_1$ and $C_2$ induced by $(E(C)\cap E(G_1))\cup \{xy\}$ and $(E(C)\cap E(G_2))\cup \{xy\}$ are cycles, respectively.
Otherwise, $G^{c_1}_1(i,\pm)$ or $G^{c_2}_2(i,\pm)$ contains a cycle, a contradiction.

By symmetry, we consider the following two cases.

If $c(x)=c(y)=i$, then $\sigma(xy)=-1$, because otherwise $C_1$ is a cycle in $G_1^{c}(i,\pm)$, a contradiction. On the other hand, since $C_1-xy$ is a path with $c(x)=c(y)=i$, the edge of which is positive iff its end-vertices are colored both by $i$ or both by $-i$, and is negative iff one of its end-vertices is colored by $i$ and the other is colored by $-i$, there is an even number of negative edges in $C_1-xy$. Hence when $\sigma(xy)=-1$, there is an odd number of negative edges in the cycle $C_1$, contradicting the fact that every cycle of $(G_1,\sigma_1)$ is positive.

If $c(x)=i$ and $c(y)=-i$, then $\sigma(xy)=1$, because otherwise $C_1$ is a cycle in $G_1^{c}(i,\pm)$, a contradiction. On the other hand, since $C_1-xy$ is a path with $c(v_i)=i$ and $c(v_j)=-i$, the edge of which is positive iff its end-vertices are colored both by $i$ or both by $-i$, and is negative iff one of its end-vertices is colored by $i$ and the other is colored by $-i$, there is an odd number of negative edges in $C_1-xy$. Hence there is an odd number of negative edges in the cycle $C_1$, contradicting the fact that every cycle of $(G_1,\sigma_1)$ is positive.

Second, we consider the case that $G_1\cap G_2=K_3:=T:=xyz$. One can image that $G_1$ is drawn in the interior of the triangle $T$ while $G_2$ is drawn in the exterior of the triangle $T$, so $xy,yz,zx$ are non-crossed in this drawing.

Actually we can prove that $G^c(i,\pm)$ is a forest for every $i\in c(u)$ and $u\in V(G)$.
Suppose, to the contrary, that $G^c(i,\pm)$ contains a cycle $C$ for some color $i$. One can easily conclude that there is an edge among $\{xy,yz,zx\}$, say $xy$, and then a path $P$ on $C$ initiated from $x$ and ended with $y$, so that $P+xy$ is a cycle in $G_1$, and meanwhile, $c(x),c(y)\in \{i,-i\}$, because otherwise $G_1^{c_1}(i,\pm)$ or $G_2^{c_2}(i,\pm)$ contains a cycle, a contradiction. At this stage, using the same arguments as the ones applied for the first major case, we would obtain a contradiction to the fact that $(G_1,\sigma_1)$ is a balanced signed graph.
\end{proof}

A plane graph is a \emph{near-triangulation} if the boundary of every face, except possibly the outer face, is a cycle on three vertices, and is \emph{triangulation} if the boundary of every face is a cycle on three vertices.
Clearly, every triangulation is near-triangulation.

\begin{thm}\label{near-triangulation-outerface}
Let $(G,\sigma)$ be a balanced signed graph and let $G$ be a near-triangulation with outer face $C=v_1v_2\ldots v_nv_1$. If there is a list $L(v)$ of colors to every vertex $v$ in $G$ so that
$L(v_1)=\{\alpha\}$, $L(v_2)=\{\beta\}$, $|L(v_k)|\geq 2$ for every $v_k\in V(C)\setminus \{v_1,v_2\}$ and $|L(u)|\geq 3$ for every $u\in V(G)\setminus V(C)$, then there is a signed tree-coloring $c$ of $(G,\sigma)$ so that $c(u)\in L(u)$ for every $u\in V(G)$.
\end{thm}

\begin{proof}
We prove it by the induction on $n$.
If there is an edge $v_iv_j\in E(G)$ with $2\leq j-i\leq n-1$, then let $(G_1,\sigma)$ be the signed graph induced by the vertices on and inside the cycle $C_1=v_iv_{i+1}\ldots v_jv_i$ and let $(G_2,\sigma)$ be the signed graph induced by the vertices on and inside the cycle $C_2=v_jv_{j+1}\ldots v_{i}v_j$. Without loss of generality, assume that $\{v_1,v_2\}\subseteq V(G_1)$. Clearly, $G_1$ and $G_2$ are near-triangulations with outer face $C_1$ and $C_2$, respectively.

Since $L(v_1)=\{\alpha\}$, $L(v_2)=\{\beta\}$, $|L(v_k)|\geq 2$ for every $v_k\in V(C_1)\setminus \{v_1,v_2\}$ and $|L(u)|\geq 3$ for every $u\in V(G_1)\setminus V(C_1)$, there is a signed tree-coloring $c_1$ of $(G_1,\sigma)$ so that $c_1(u)\in L(u)$ for every $u\in V(G_1)$ by the induction hypothesis. Now $v_i$ and $v_j$ have been colored with $\alpha':=c(v_i)$ and $\beta':=c(v_j)$, respectively.

Define a list $L'(v)$ of colors to every vertex $v$ in $G_2$ so that $L'(v_i)=\{\alpha'\}$, $L'(v_j)=\{\beta'\}$ and $L'(u)=L(u)$ for every $u\in V(G_2)\setminus \{v_i,v_j\}$. Since
$|L'(v_k)|=|L(v_k)|\geq 2$ for every $v_k\in V(C_2)\setminus \{v_i,v_j\}$ and $|L'(u)|=|L(u)|\geq 3$ for every $u\in V(G_2)\setminus V(C_2)$, there is a signed tree-coloring $c_2$ of $(G_2,\sigma)$ so that $c_2(u)\in L'(u)$ for every $u\in V(G_2)$.

Combining $c_1$ with $c_2$, we obtain a signed tree-coloring $c$ of $(G,\sigma)$ so that $c(u)\in L(u)$ for every $u\in V(G)$ by Lemma \ref{comb-lem}.

Hence in the following we assume that there is no such an edge $v_iv_j\in E(G)$ with $2\leq j-i\leq n-1$. 
Let $v_1,u_1,u_2,\ldots,u_t,v_{n-1}$ be neighbors of $v_n$ in cyclic order around $v_n$. By the previous assumption, we have $v_1u_1,u_1u_2,\ldots,u_{t-1}u_t,u_tv_{n-1}\in E(G)$, since $G$ is a near-triangulation. Hence $G^*=G-v_n$ is a near-triangulation with outer face $C=v_1v_2\ldots v_{n-1}u_t\ldots u_2u_1v_1$.
Let $\gamma$ be a color in $L(v_n)$ that is different from $\alpha\sigma(v_1v_n)$. Define a list $L^*(v)$ of colors to every vertex $v$ in $G^*$ so that $L^*(u_i)=L(u_i)\setminus \{\gamma\sigma(v_nu_i)\}$ for every $1\leq i\leq t$ and $L^*(v)=L(v)$ for every $v\in V(G^*)\setminus \{u_1,u_2,\ldots, u_t\}$. Since $|L(u_i)|\geq 3$ for every $1\leq i\leq t$, $|L^*(u_i)|\geq 2$ for every $1\leq i\leq t$. Hence we can apply the induction hypothesis to $G^*$ and conclude that there is a signed tree-coloring $c^*$ of $(G^*,\sigma)$ so that $c^*(u)\in L^*(u)\subseteq L(u)$ for every $u\in V(G^*)$. Now we complete the coloring $c$ of $(G,\sigma)$ by coloring $v_n$ with $\gamma$.
Since $\gamma\neq c^*(v_1)\sigma(v_1v_n)$ and $\gamma\neq c^*(u_i)\sigma(v_nu_i)$ for every $1\leq i\leq t$, $G^c(\gamma,\pm)$ is a forest.
Note that we would not mind whether $\gamma$ is equal to $c^*(v_{n-1})\sigma(v_{n-1}v_n)$ or not.
Hence $c$ is a signed tree-coloring of $(G,\sigma)$ so that $c(u)\in L(u)$ for every $u\in V(G)$.
\end{proof}

Actually, Theorem \ref{near-triangulation-outerface} directly implies the following

\begin{cor}
If $(G,\sigma)$ is a balanced signed triangulation, then $va(G,\sigma)\leq 3$. \hfill$\square$
\end{cor}

Since all-positive signed graph is balanced and the signed tree-coloring for such a graph $(G,\sigma)$ is equivalent to the tree-coloring for the unsigned graph $G$, the signed vertex arboricity of $(G,\sigma)$ is the same with the (unsigned) vertex arboricity of $G$. Hence we have

\begin{cor}
If $G$ is a triangulation, then $va(G)\leq 3$. \hfill$\square$
\end{cor}

Specially, since every planar graph is a subgraph of a triangulation and $va(G')\leq va(G)$ if $G'\subseteq G$, we deduce a well-known result contributed by Chartrand, Kronk and Wall \cite{CKW}.

\begin{cor}\label{va-planar}
If $G$ is a planar graph, then $va(G)\leq 3$. \hfill$\square$
\end{cor}


Now let us focus on the class of $K_5$-minor-free graphs, which is a larger class than the one of planar graphs. Before showing the second main result of this paper, we present some additional useful lemmas.

\begin{lem}\cite{Wagner}\label{K5-minor-free-structure}
If $G$ is an edge-maximal $K_{5}$-minor-free graph, then $G=G_1\cup G_2$, where $G_1$ is an edge-maximal $K_{5}$-minor-free graph, $G_2$ is an edge-maximal planar graph (i.e., triangulation) or the Wagner graph (see Figure \ref{wgph}), and $G_1\cap G_2=K_2$ or $K_3$.
\end{lem}

\begin{lem}\label{near-triangulation-outerface-K3}
Let $(G,\sigma)$ be a balanced signed graph and let $G$ be a near-triangulation. If there is a triangle $T:=xyzx$ in $G$ and a list $L(v)$ of colors to every vertex $v$ in $G$ so that $L(x)=\{\alpha\}$, $L(y)=\{\beta\}$, $L(z)=\{\gamma\}$,
$|L(v)|\geq 3$ for every $v\in V(G)\setminus  \{x,y,z\}$ and $T$ has a signed tree-coloring by coloring $x,y,z$ with $\alpha,\beta,\gamma$, respectively,
then there is a signed tree-coloring $c$ of $(G,\sigma)$ so that $c(v)\in L(v)$ for every $v\in V(G)$.
\end{lem}

\begin{proof}
We prove it by the induction on the order of $G$.
If both the interior and exterior of the triangle $T:=xyzx$ contains vertices of $G$, then let $(G_1,\sigma)$ denote the graph induced by the vertices inside or on $T$, and let $(G_2,\sigma)$ denote the graph induced by the vertices outside or on $T$. Clearly, we can apply the induction hypothesis to both $G_1$ and $G_2$, and then obtain signed tree-colorings $c_1$ and $c_2$ of $(G_1,\sigma)$ and $(G_2,\sigma)$ so that $c(v)\in L(v)$ for every $v\in V(G_1)$ and $v\in V(G_2)$, respectively.

Combining $c_1$ with $c_2$, we obtain a signed tree-coloring $c$ of $(G,\sigma)$ so that $c(v)\in L(v)$ for every $v\in V(G)$ by Lemma \ref{comb-lem}.

Hence we assume, without loss of generality, that $T:=xyzx$ is the outer face of $G$. Let $x,u_1,u_2,\ldots,u_t,y$ be neighbors of $z$ in cyclic order around $z$. Since $G$ is a near-triangulation, $xu_1,u_1u_2,\ldots,u_{t-1}u_t,u_ty\in E(G)$. Color $z$ with $\gamma$ and define a list $L^*(v)$ of colors to every vertex $v$ in $G^*:=G-z$ so that $L^*(u_i)=L(u_i)\setminus \{\gamma\sigma(zu_i)\}$ for every $1\leq i\leq t$ and $L^*(v)=L(v)$ for every $v\in V(G^*)\setminus \{u_1,u_2,\ldots,u_t\}$. Since $L^*(x)=\{\alpha\},L^*(y)=\{\beta\}$, $|L^*(u_i)|\geq 3-1=2$ for every $1\leq i\leq t$ and $|L^*(v)|\geq 3$ for every $v\in V(G^*)\setminus \{u_1,u_2,\ldots,u_t\}$,
there is a signed tree-coloring $c^*$ of $(G^*,\sigma)$ so that $c^*(v)\in L^*(v)$ for every $v\in V(G^*)$ by Theorem \ref{near-triangulation-outerface}. Combining $c^*$ with the color on $z$, we obtain a coloring $c$ of $(G,\sigma)$ so that $c(v)\in L(v)$ for every $v\in V(G)$. Next, we show that $c$ is a signed tree-coloring of $(G,\sigma)$.

In fact, $G^{c}(\xi,\pm)$ is $G^{c^*}(\xi,\pm)$ for any $\xi\neq \pm \gamma$, and thus is a forest, because $c^*$ is a signed tree-coloring. Hence we just claim that $G^{c}(\gamma,\pm)$ is a forest.

If $\gamma\neq \alpha\sigma(zx)$ or $\gamma\neq \beta\sigma(yz)$, then $G^{c}(\gamma,\pm)$ is a forest, since
$r\neq c^*(u_i)\sigma(zu_i)$ for every $1\leq i\leq t$. If $\gamma=\alpha\sigma(zx)=\beta\sigma(yz)$, then by symmetry we consider three cases according to the fact that $T$ is a positive cycle.
If $\sigma(xy)=\sigma(yz)=\sigma(zx)=1$, then $\alpha=\beta=\gamma$, and $T=T^c(\gamma,\pm)$. If $\sigma(xy)=1$ and $\sigma(yz)=\sigma(zx)=-1$, then $\alpha=\beta=-\gamma$, and $T=T^c(-\gamma,\pm)$. If $\sigma(yz)=1$ and $\sigma(xy)=\sigma(zx)=-1$, then $\alpha=-\beta=-\gamma$, and $T=T^c(\gamma,\pm)$. In each case
$T$ does not have the signed tree-coloring as mentioned in the lemma, a contradiction.
\end{proof}

\begin{figure}
  \centering
  \includegraphics[width=5cm]{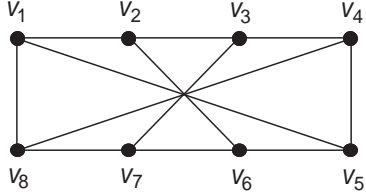}
  \caption{Wagner graph}\label{wgph}
\end{figure}

\begin{lem}\label{wagnerextended}
  Let $(W,\sigma)$ be a signed Wagner graph (see Figure \ref{wgph}). If there is a list $L(v_i)$ of colors to every vertex $v_i$ in $W$ so that $L(v_i)=\{\alpha\}$ and $L(v_j)=\{\beta\}$ for en edge $v_iv_j\in E(G)$, and
  $|L(v_k)|\geq 3$ for every $k\neq i,j$, then there is a signed tree-coloring $c$ of $(W,\sigma)$ so that $c(v_k)\in L(v_k)$ for every $1\leq k\leq 8$.
\end{lem}

\begin{proof}
By symmetry, considering two speical cases here is enough. First, if $i=1$ and $j=2$, then
let $c(v_1)=\alpha$, $c(v_2)=\beta$, and choose $c(v_3)\in L(v_3)\setminus \{c(v_2)\sigma(v_2v_3)\}$, $c(v_4)\in L(v_4)\setminus \{c(v_3)\sigma(v_3v_4)\}$, $c(v_5)\in L(v_5)\setminus \{c(v_1)\sigma(v_1v_5),c(v_4)\sigma(v_4v_5)\}$, $c(v_6)\in L(v_6)\setminus \{c(v_2)\sigma(v_2v_6),c(v_5)\sigma(v_5v_6)\}$, $c(v_7)\in L(v_7)\setminus \{c(v_3)\sigma(v_3v_7),c(v_6)\sigma(v_6v_7)\}$, $c(v_8)\in L(v_8)\setminus \{c(v_4)\sigma(v_4v_8),c(v_7)\sigma(v_7v_8)\}$.
It is easy to check that the resulting coloring $c$ is a signed tree-coloring as required. Note that we may have $\beta=\alpha\sigma(v_1v_2)$ or $c(v_8)=\alpha\sigma(v_1v_8)$ by the above choice of coloring but it does not matter at all.
Second, if $i=1$ and $j=5$, then we can similarly solve it, so omit the proof.
\end{proof}

\begin{thm}\label{thm-K5-minor-free}
If $(G,\sigma)$ is an edge-maximal $K_5$-minor-free graph with balanced signature and a list $L(v)$ of colors is given to every vertex $v$ in $G$ so that $|L(v)|\geq 3$,
then there is a signed tree-coloring $c$ of $(G,\sigma)$ so that $c(v)\in L(v)$ for every $v\in V(G)$. Specially, if we choose $L(v)=M_3=\{-1,0,1\}$ for every $v\in V(G)$, then the previous coloring is a signed tree-3-coloring of $(G,\sigma)$ and thus the signed vertex arboricity of $(G,\sigma)$ is at most 3, i.e., $va(G,\sigma)\leq 3$.
\end{thm}

\begin{proof}
We prove it by the induction on the order of $(G,\sigma)$. By Lemma \ref{K5-minor-free-structure}, $G=G_1\cup G_2$, where $G_1$ is an edge-maximal $K_{5}$-minor-free graph, $G_2$ is an edge-maximal planar graph (i.e., triangulation) or the Wagner graph, and $H:=G_1\cap G_2=K_2$ or $K_3$.

If $H=K_2$, then let $V(H)=\{x,y\}$. Since $G_1$ is an edge-maximal $K_{5}$-minor-free graph with smaller order, by the induction hypothesis, $G_1$ has a signed tree-coloring $c_1$ such that $c_1(v)\in L(v)$ for every $v\in V(G_1)$. Now, the vertices $x$ and $y$ have been colored, and then by Theorem \ref{near-triangulation-outerface}, there is a signed tree-coloring $c_2$ of $G_2$ such that $c_2(v)\in L(v)$ for every $v\in V(G_2)$ and $c_2(x)=c_1(x)$, $c_2(y)=c_1(y)$.
Combining $c_1$ with $c_2$, we obtain a signed tree-coloring $c$ of $G$ such that $c(v)\in L(v)$ for every $v\in V(G)$ by Lemma \ref{comb-lem}.

If $H=K_3$, then $G_2$ is an edge-maximal planar graph (i.e., triangulation), because Wagner graph does not contain triangles.
Let $V(H)=\{x,y,z\}$. Since $G_1$ is an edge-maximal $K_{5}$-minor-free graph with smaller order, by the induction hypothesis, $G_1$ has a signed tree-coloring $c_1$ such that $c_1(v)\in L(v)$ for every $v\in V(G_1)$. Now, the vertices $x,y$ and $z$ have been colored so that $T$ admits a signed tree-coloring, thus
by Lemma \ref{near-triangulation-outerface-K3}, there is a signed tree-coloring $c_2$ of $G_2$ such that $c_2(v)\in L(v)$ for every $v\in V(G_2)$ and $c_2(x)=c_1(x)$, $c_2(y)=c_1(y)$, $c_2(z)=c_1(z)$.
Combining $c_1$ with $c_2$, we obtain a signed tree-coloring $c$ of $G$ such that $c(v)\in L(v)$ for every $v\in V(G)$ by Lemma \ref{comb-lem}.
\end{proof}

Note that all-positive signed graph is balanced and its signed vertex arboricity is the same with its (unsigned) vertex arboricity, and that $va(G')\leq va(G)$ if $G'$ is a subgraph of an unsigned graph $G$.
Hence Theorem \ref{thm-K5-minor-free} directly deduces the following corollary.

\begin{cor}
If $G$ is a $K_5$-minor-free graph, then $va(G)\leq 3$.
\end{cor}

Again, this corollary generalizes Corollary \ref{va-planar} and thus the bound 3 for $va(G)$ here cannot be lowered (see \cite{CKW} for a planar graph with vertex arboricity exactly 3). Therefore, the bound 3 for $va(G,\sigma)$ in Theorem \ref{thm-K5-minor-free} is sharp.

\end{document}